\documentclass{amsart}

\usepackage{amssymb}

\usepackage{amsfonts}
\usepackage{mathrsfs}
\usepackage{bbm}

\usepackage[dvipdfmx, dvipsnames, svgnames, x11names]{xcolor} %

\theoremstyle{plain}
    \newtheorem{theorem}{Theorem}[section]
    \newtheorem{lemma}[theorem]{Lemma}
    \newtheorem{proposition}[theorem]{Proposition}
    \newtheorem{corollary}[theorem]{Corollary}
\theoremstyle{definition}
    
    \newtheorem{remark}{Remark}[section]
    
    \newtheorem*{acknowledgement}{Acknowledgement}
\theoremstyle{remark}
    
\numberwithin{equation}{section}

\newcommand{\cleq}{\lesssim}

\def\norm#1{\left\Vert #1 \right\Vert} 

\DeclareMathOperator{\re}{Re}

\begin{document}

\title[Asymptotic order for linear wave eq. with scale-invariant damping]{Remarks on asymptotic order for the linear wave equation with the scale-invariant damping and mass with $L^r$-data}
\author[T. Inui]{Takahisa Inui}
\address{Department of Mathematics, Graduate School of Science, Osaka University, Toyonaka, Osaka 560-0043, Japan}
\email{inui@math.sci.osaka-u.ac.jp}
\author[H. Mizutani]{Haruya Mizutani}
\address{Department of Mathematics, Graduate School of Science, Osaka University, Toyonaka, Osaka 560-0043, Japan}
\email{haruya@math.sci.osaka-u.ac.jp}
\date{\today}
\keywords{wave equation, scaling invariant damping, scattering, asymptotic order, $L^r$-data, $\dot{H}^{-\gamma}$-data}
\subjclass[2010]{35L05, 35B40, 47A40, etc.}
\maketitle

\begin{abstract}
%
%
In the present paper, we consider the linear wave equation with the scale-invariant damping and mass. It is known that the global behavior of the solution depends on the size of the coefficients in front of the damping and mass at initial time $t=0$. Indeed, the solution satisfies the similar decay estimate to that of the corresponding heat equation if it is large and to that of the modified wave equation if it is small. In our previous paper, we obtain the scattering result and its asymptotic order for the data in the energy space $H^1\times L^2$ when the coefficients are in the wave regime. In fact, the threshold of the coefficients relies on the spatial decay of the initial data. Namely, it varies depending on $r$ when the initial data is in $L^r$ ($1\leq r < 2$). In the present paper, we will show the scattering result and the asymptotic order in the wave regime for $L^r$-data, which is wider than the wave regime for the data  in the energy space. 
Moreover, we give an improvement of the asymptotic order obtained in our previous paper for the data  in the energy space.
\end{abstract}

\tableofcontents

\section{Introduction}

\subsection{Motivation}
We consider the linear wave equation with the scale-invariant damping and mass:
\begin{align}
\label{DW}
\tag{DW}
	\begin{cases}
	\displaystyle \partial_t^2 u - \Delta u +\frac{\mu_1}{1+t} \partial_t u + \frac{\mu_2}{(1+t)^2} u =0, & (t, x) \in (0,\infty) \times \mathbb{R}^d,
	\\
	(u(0), \partial_t u(0)) = (u_0, u_1), &  x \in \mathbb{R}^d,
	\end{cases}
\end{align}
where $\mu_1,\mu_2 \in \mathbb{R}$, $d\in \mathbb{N}$, and $(u_0, u_1) \in H^1(\mathbb{R}^d)\times L^2(\mathbb{R}^d)$.

Before considering our linear equation, we consider the following wave equation with time-dependent damping:
\begin{align}
\label{eq1.0}
	\partial_t^2 w -\Delta w +\frac{b_0}{(1+t)^{\beta}} \partial_t w = 0,
	\quad (t,x) \in (0,T) \times \mathbb{R}^d,
\end{align}
where $b_0>0$ and $\beta \in \mathbb{R}$. 
Wirth \cite{Wir04,Wir06,Wir07_1} classify it into four cases from the viewpoint of the global behavior of the solutions. 
The equation is called overdamping when $\beta <-1$. In this case, the solution does not decay to zero as $t \to \infty$. 
When $-1 \leq \beta <1$, the damping term is called effective. The solution in this case behaves like one of the heat equation $\frac{b_0}{(1+t)^{\beta}} \partial_t h-\Delta h  = 0$. If $\beta >1$, then it is known that the solution scatters to one of the free wave equation. Then, we say that we have scattering. If $\beta =1$, then the equation is invariant under the scaling
\begin{align*}
	w^{\sigma}(t,x):= w( \sigma (1+t)-1,\sigma x) \text{ for } \sigma>0.
\end{align*}
Namely, $w^{\sigma}$ is a solution if $w$ is a solution. Thus, this case is said to be scale-invariant. 
\begin{table}[htb]
  \begin{tabular}{|l|l|} \hline
    $\beta\in (-\infty,-1)$ & overdamping  \\ \hline
    $\beta\in [-1,1)$ & effective  \\ \hline
    $\beta =1$ & scale-invariant  \\ \hline
    $\beta\in (1,\infty)$ & scattering  \\ \hline
  \end{tabular}
\end{table}

In the scale-invariant case, i.e. $\beta=1$, Wirth \cite{Wir04} showed that the solution satisfies the following $L^p$-$L^q$ estimates:
\begin{align*}
	\|w(t)\|_{L^q}
	\lesssim 
	\begin{cases}
	(1+t)^{\max\{ -\frac{d-1}{2}\left(\frac{1}{p}-\frac{1}{q}\right)-\frac{b_0}{2}, -d\left(\frac{1}{p}-\frac{1}{q}\right)+1-b_0\}}, 
	&\text{ if } b_0\in (0,1),
	\\
	(1+t)^{\max\{ -\frac{d-1}{2}\left(\frac{1}{p}-\frac{1}{q}\right)-\frac{b_0}{2},-d\left(\frac{1}{p}-\frac{1}{q}\right)\}},
	&\text{ if } b_0\in (1,\infty),
	\end{cases} 
\end{align*}
where $1<p\le 2$, $1/p+1/q=1$, the implicit constant depends on $\|w_{l,0}\|_{H^s_p}+\|w_{l,1}\|_{H_p^{s-1}}$, and  $s=d(1/p-1/q)$.
This result means that the solution behaves like that of the corresponding heat equation 
$\frac{b_0}{1+t} \partial_t h-\Delta h = 0$ if $b_0$ is sufficiently large and  the solution behaves like that of the wave equation if $b_0$ is sufficiently small. 
Therefore, the scale-invariant case is critical in the sense of the global behavior of the solutions. Especially, the constant $b_0$ plays an important role to determine the global behavior of the solutions.

In the previous paper \cite{InMi20}, we considered the linear wave equation with the scale-invariant damping and mass \eqref{DW}, showed the scattering result (see also \cite{Wir06,Wir07_2,NaWi15}), and obtained the asymptotic order. Namely, there exists
$\overrightarrow{v_{+}} \in \dot{H}^1(\mathbb{R}^d) \times L^2(\mathbb{R}^d)$ such that
\begin{align*}
	&\norm{\overrightarrow{u}(t) -(1+t)^{-\frac{\mu_1}{2}} \mathcal{W}(t)\overrightarrow{v_{+}} }_{\dot{H}^1\times L^2}
	\\
	&\cleq
	(\mu+|\mu_1|)(\norm{u_0}_{L^2}+\norm{u_1}_{L^2})
	\begin{cases}
	(1+t)^{-\frac{1}{2} + \re\nu - \frac{\mu_1}{2}}, & \mu \neq 1/4,
	\\
	(1+t)^{-\frac{1}{2} - \frac{\mu_1}{2}}\{1+\log(1+t)\}, & \mu =1/4,
	\end{cases}
\end{align*}
provided that $\mu \geq 0$, where $\mathcal{W}(t)$ is the solution map of the free wave equation, $\overrightarrow{u}(t)=(u(t),\partial_t u(t))$, 
\begin{align}
\label{eq1.1}
	\mu:=\frac{\mu_1(2-\mu_1)}{4} +\mu_2, \quad  \text{ and }\quad  \nu:=
	\begin{cases}
	\frac{1}{2}\sqrt{1-4\mu} & (\mu\leq 1/4)
	\\
	\frac{i}{2}\sqrt{4\mu-1} & (\mu > 1/4)
	\end{cases}.
\end{align}
To obtain this result, we used the so-called Liouville transform $v := (1+t)^{\frac{\mu_1}{2}} u$ and transformed \eqref{DW} into the Klein-Gordon equation with the scale-invariant mass. After that, we applied an $L^2$-estimate of the solution obtained by \cite{BoRe12}. 

The first aim of the present paper is to improve the asymptotic order by assuming that the initial data belongs to not only $H^1 \times L^2$ but also an additional function space. 
Recently, do Nascimento, Palmieri, and Reissig \cite{NPR17} gave a better estimate with respect to time $t$
of the $L^2$-norm of the solution than that in \cite{BoRe12} when the initial data belongs to $L^r(\mathbb{R}^d)$ for $1\leq r <2$. Thus, if the initial data belongs to $L^r(\mathbb{R}^d)$, then we expect to improve the asymptotic order combining our previous argument in \cite{InMi20} and the estimate in \cite{NPR17}. We will also get the scattering result in the case which was not treated in the previous paper  \cite{InMi20}. For example, we can get the scattering in the case of $\mu_1>2$ and $\mu_2=0$, which implies $\mu<0$. 

The second aim in the present paper is to improve our previous asymptotic order for $H^1 \times L^2$-data. By the density, the initial data in $H^1 \times L^2$ can be approximated by an $L^r$-data. This density and the above argument for $L^r$-data imply an additional small order estimate. This is the so-called invisible estimate. The estimate looks like the invisible decay for the solution to the heat equation. Such estimate was obtained for \eqref{eq1.0} with $b_0>2$ by Hirosawa and Nakazawa \cite[Theorem 1.2]{HiNa03}. They showed  $\lim_{t\to \infty} (1+t)\| \overrightarrow{w}(t) \|_{\dot{H}^1 \times L^2} =0$ and $\lim_{t \to \infty} \|w(t)\|_{L^2} =0$ by an energy method. We consider \eqref{DW} and give an extension of their result by the density and the better estimate for $L^r$-data. This extension also gives an improvement of the asymptotic order which was obtained by our previous paper \cite{InMi20} for $H^1\times L^2$-data. 

\subsection{Main results}


As in the previous paper \cite{InMi20}, by the Liouville transfrom $v := (1+t)^{\frac{\mu_1}{2}} u$, we transform \eqref{DW} into  the following Klein-Gordon equation with the scale-invariant mass:
\begin{align}
\label{KG}
\tag{KG}
	\begin{cases}
	\displaystyle \partial_t^2 v - \Delta v +\frac{\mu}{(1+t)^2}v =0, & (t, x) \in (0,\infty) \times \mathbb{R}^d,
	\\
	(v(0), \partial_t v(0)) = (v_0, v_1), &  x \in \mathbb{R}^d,
	\end{cases}
\end{align}
where we set $(v_0,v_1):=(u_0, \mu_1u_0/2+u_1)$ and $\mu$ is in \eqref{eq1.1}. 
In what follows, we set $\overrightarrow{v}(t)=(v(t),\partial_t v(t))$, $\mathcal{W}(t)$ denotes the solution map of the free wave equation (see Section \ref{sec2} below), and $(\mathcal{W}(t)\overrightarrow{v_{+}})_{j}$ denotes the $j$-th component of $\mathcal{W}(t)\overrightarrow{v_{+}}$ for $j=1,2$. 

\subsubsection{Asymptotic order for $L^r$-data}

We define the function space $\Sigma^{r}$ 
by
\begin{align*}
	\Sigma^{r}&:=(H^1(\mathbb{R}^d) \cap L^r(\mathbb{R}^d)) \times( L^2(\mathbb{R}^d) \cap L^r(\mathbb{R}^d)),
\end{align*}
for $1\leq r < 2$. 
Moreover, we set
\begin{align*}
	\alpha&=\alpha(r,\mu):=\frac{1}{2} + \re\nu -\frac{d(2-r)}{2r},
	\\
	\delta & := 
	\begin{cases}
	1, & \mu =1/4,
	\\
	0, & \mu\neq 1/4,
	\end{cases}
\end{align*}
where $\nu$ is defined in \eqref{eq1.1}. 


We obtain the following scattering and asymptotic order for the initial data in $\Sigma^{r}$. 

\begin{theorem}
\label{thm1.0}
Let $r \in [1,2)$ and $v$ be the solution of \eqref{KG}. Assume $(v_0,v_1)\in \Sigma^{r}$ and that $\mu$ and $r$ satisfy $\alpha<1$. 
Then, there exists $\overrightarrow{v_{+}} \in \dot{H}^1(\mathbb{R}^d) \times L^2(\mathbb{R}^d)$ such that
\begin{align*}
	&\norm{\overrightarrow{v}(t) - \mathcal{W}(t)\overrightarrow{v_{+}}}_{\dot{H}^1 \times L^2}
	\\
	&\cleq |\mu| \norm{ (v_0,v_1) }_{\Sigma^{r}}
	\begin{cases}
	(1+t)^{-1}\{1+\log(1+t)\}^{\delta} & \text{ if } \alpha < 0, \text{ or } \alpha=0 \text{ and } r>1,
	\\
	(1+t)^{-1}\{1+\log(1+t)\}^{\frac{1}{2}+\delta}& \text{ if } \alpha = 0 \text{ and } r=1,
	\\
	(1+t)^{\alpha-1}\{1+\log(1+t)\}^{\delta} & \text{ if } \alpha \in( 0,1),
	\end{cases}
\end{align*}
where the implicit constant does not depend on time and the initial data $(v_0,v_1)$.
\end{theorem}

\begin{remark}
We can treat the case $\mu<0$ as long as $\alpha<1$. One can find  $\overrightarrow{v}_{+} \neq 0$ in general when $\mu<0$ (see Proposition \ref{propA.1} below). 
\end{remark}

By retransforming $u=(1+t)^{-\mu_1/2}v$, we have the following result for \eqref{DW}. 
\begin{corollary}
\label{cor1.0}
Let $r \in [1,2)$ and $u$ be the solution of \eqref{DW}. Assume $(u_0,u_1)\in \Sigma^{r}$ and that $\mu$ and $r$ satisfy $\alpha<1$. 
 Then, there exists $\overrightarrow{v_{+}} \in \dot{H}^1(\mathbb{R}^d) \times L^2(\mathbb{R}^d)$ such that
\begin{align*}
	&\norm{\overrightarrow{u}(t) -(1+t)^{-\frac{\mu_1}{2}} (\mathcal{W}(t)\overrightarrow{v_{+}})_1 }_{\dot{H}^1}
	\\
	&\cleq |\mu| \norm{ (v_0,v_1) }_{\Sigma^{r}}
	\begin{cases}
	(1+t)^{-1 - \frac{\mu_1}{2}}\{1+\log(1+t)\}^{\delta} &\text{ if }   \alpha<0, \text{ or } \alpha=0 \text{ and } r>1,
	\\
	(1+t)^{ -1- \frac{\mu_1}{2}}\{1+\log(1+t)\}^{\frac{1}{2}+\delta}
	& \text{ if }   \alpha=0  \text{ and } r=1,
	\\
	(1+t)^{\alpha-1- \frac{\mu_1}{2}} \{1+\log(1+t)\}^{\delta}&\text{ if }    \alpha \in (0,1)
	\end{cases}
\end{align*}
and
\begin{align*}
	&\norm{\overrightarrow{u}(t) -(1+t)^{-\frac{\mu_1}{2}} (\mathcal{W}(t)\overrightarrow{v_{+}})_2 }_{L^2}
	\\
	&\cleq (|\mu| +|\mu_1|) \norm{ (v_0,v_1) }_{\Sigma^{r}}
	\\
	&\quad \times 
	\begin{cases}
	(1+t)^{-1 - \frac{\mu_1}{2}}\{1+\log(1+t)\}^{\delta} &\text{ if }   \alpha<0, \text{ or } \alpha=0 \text{ and } r>1,
	\\
	(1+t)^{ -1- \frac{\mu_1}{2}}\{1+\log(1+t)\}^{\frac{1}{2}+\delta}
	& \text{ if }   \alpha=0  \text{ and } r=1,
	\\
	(1+t)^{\alpha-1- \frac{\mu_1}{2}} \{1+\log(1+t)\}^{\delta}&\text{ if }    \alpha \in (0,1).
	\end{cases}
\end{align*}
\end{corollary}

\begin{remark}
It is worth emphasizing that the assumption in Corollary \ref{cor1.0} is satisfied when $0<\mu_1<d+2$, $\mu_2=0$, and $r=1$.\footnote{It is not needed to assume $\mu_1>0$. Theorem \ref{thm1.0} and Corollary \ref{cor1.0} hold even when $\mu_1$ is negative. For example, those are valid when $-d<\mu_1<d+2$, $\mu_2=0$, and $r=1$.} D'Abbicco \cite[Theorem 1]{Dab15} showed the following estimate.
\begin{align*}
	\| u(t) \|_{\dot{H}^1} \cleq 
	\begin{cases}
	(1+t)^{-\frac{d}{2}} & \text{ if } \mu_1 > d+2
	\\
	(1+t)^{-\frac{\mu_1}{2}} \{1+\log(1+t)\} & \text{ if } \mu_1 = d+2
	\\
	(1+t)^{-\frac{\mu_1}{2}}& \text{ if } \mu_1 < d+2
	\end{cases}
\end{align*}
when $\mu_1>1$ and $\mu_2=0$, where the implicit constant depends on the $(H^1\cap L^1) \times ( L^2 \cap L^1)$-norm of the initial data. This estimate shows that  the solution satisfies the similar decay estimate to the corresponding heat equation if $\mu_1>d+2$ and to the modified wave equation if $\mu_1 <d+2$. 
Our result, Corollary \ref{cor1.0}, means that the solution of \eqref{DW} scatters to one of the modified wave equation when $\mu_1<d+2$, $\mu_2=0$, and the initial data belongs to $L^1$. 
\end{remark}

\begin{remark}
We also obtain similar results to Theorem \ref{thm1.0} and Corollary \ref{cor1.0} for the initial data in a negative homogeneous Sobolev space. In fact, the results partially cover Theorem \ref{thm1.0} and Corollary \ref{cor1.0}. See Appendix \ref{appB} for the details. 
\end{remark}

\subsubsection{Improvement of the asymptotic order for $H^1\times L^2$-data}

We have the following estimates for $\dot{H}^1\times L^2$-norm and $L^2$-norm.  

\begin{theorem}
\label{thm1.5}
Let $(v_0,v_1) \in H^1(\mathbb{R}^d) \times L^2(\mathbb{R}^d)$ and $v$ be the solution of \eqref{KG}. 
If $\mu<0$, then we have
\begin{align}
\label{eq01}
	\lim_{t\to \infty} (1+t)^{\frac{1}{2} -\re \nu} \| \overrightarrow{v}(t) \|_{\dot{H}^1 \times L^2} =0.
\end{align}
Moreover, for any $\mu\in \mathbb{R}$, we have
\begin{align}
\label{eq02}
	\begin{cases}
	\lim_{t\to \infty} (1+t)^{-\frac{1}{2} -\re \nu} \| v(t) \|_{L^2} =0 & \text{ if } \mu\neq1/4,
	\\
	\lim_{t\to \infty} (1+t)^{-\frac{1}{2}}\{1+\log(1+t)\}^{-1} \| v(t) \|_{L^2} =0 & \text{ if } \mu=1/4.
	\end{cases}
\end{align}
Thus, letting $(u_0,u_1)\in H^1(\mathbb{R}^d) \times L^2(\mathbb{R}^d)$ and $u$ be the solution of \eqref{DW}, we also have
\begin{align}
\notag
	&\lim_{t\to \infty} (1+t)^{\frac{1}{2} -\re \nu +\frac{\mu_1}{2}}  \| \overrightarrow{u}(t) \|_{\dot{H}^1 \times L^2} =0 
	\quad \text{ if } \mu<0,
	\\
\label{eq03}
	&\begin{cases}
	\lim_{t\to \infty} (1+t)^{-\frac{1}{2} -\re \nu+\frac{\mu_1}{2}}  \| u(t) \|_{L^2}  =0 & \text{ if } \mu\neq1/4,
	\\
	\lim_{t\to \infty} (1+t)^{-\frac{1}{2}+\frac{\mu_1}{2}}\{1+\log(1+t)\}^{-1}   \| u(t) \|_{L^2} =0 & \text{ if } \mu=1/4.
	\end{cases}
\end{align}
\end{theorem}

\begin{remark}
\ 
\begin{enumerate}
\item If $\mu_1>2$ and $\mu_2=0$, then $\mu<0$. Moreover, in this case, we have $\frac{1}{2} -\re \nu + \frac{\mu_1}{2} =1$ and $-\frac{1}{2} -\re \nu +\frac{\mu_1}{2}=0$. Therefore, the results (1.7) and (1.8) in \cite[Theorem 1.2]{HiNa03} are included in Theorem \ref{thm1.5}. 
\item We note that we only assume $(v_0,v_1) \in H^1 \times L^2$ in Theorem \ref{thm1.5}. If we assume $(v_0,v_1)  \in \Sigma^{r}$ for some $r \in [1,2)$, then we can obtain the explicit estimate as in Lemma \ref{lem2}. 
\item In the previous paper \cite[Corollary A.9]{InMi20}, we only showed that $\| \overrightarrow{v}(t) \|_{\dot{H}^1 \times L^2} \cleq (1+t)^{-\frac{1}{2} +\re \nu}$ and $\| v(t) \|_{L^2} \cleq (1+t)^{\frac{1}{2} +\re \nu}$ when $\mu<0$. Theorem \ref{thm1.5} gives better estimates. 

\item do Nascimento and Wirth showed that $\lim_{t\to \infty} (1+t)^{-1+\frac{\mu_1}{2}} \| u(t) \|_{L^2}  =0$ for $L^2 \times H^{-1}$-data in \cite[Remark 3.4]{NaWi15}. Theorem \ref{thm1.5} means that 
the better estimate \eqref{eq03} than this estimate 
holds for $H^1\times L^2$-data. 
\end{enumerate}
\end{remark}

Theorem \ref{thm1.5} gives the following improvement of our previous asymptotic order. 

\begin{corollary}
\label{cor1.6}
Let $\mu >0$. Then, there exists $\overrightarrow{v_{+}} \in \dot{H}^1(\mathbb{R}^d) \times L^2(\mathbb{R}^d)$ such that
\begin{align*}
	&\lim_{t \to \infty} (1+t)^{\frac{1}{2}-\re\nu}\| \overrightarrow{v}(t) - \mathcal{W}(t)\overrightarrow{v_{+}}\|_{\dot{H}^1 \times L^2}=0 
	& \text{ if } \mu \neq1/4,
	\\
	&\lim_{t \to \infty} (1+t)^{\frac{1}{2}} \{1+\log(1+t)\}^{-1}\| \overrightarrow{v}(t) - \mathcal{W}(t)\overrightarrow{v_{+}}\|_{\dot{H}^1 \times L^2}=0 
	& \text{ if } \mu = 1/4.
\end{align*}
\end{corollary}

The result in Corollary \ref{cor1.6} is better than the asymptotic order in  \cite[Theorem 1.1]{InMi20}, that is, we get small order. 
Corollary \ref{cor1.6} implies the estimate for \eqref{DW} by the retransformation. However, we omit the statement. 

\section{Preliminaries}
\label{sec2}

\subsection{Expression of the linear equation}

We consider the linear Klein-Gordon equation with the scale-invariant mass with the initial data given at $t=t_0\geq 0$:
\begin{align}
\label{eq2.1}
	\begin{cases}
	\displaystyle\partial_t^2 v - \Delta v +\frac{\mu}{(1+t)^2}v =0, & (t, x) \in (t_0,\infty) \times \mathbb{R}^d,
	\\
	(v(t_0), \partial_t v(t_0)) = (v_0, v_1), &  x \in \mathbb{R}^d.
	\end{cases}
\end{align}
We recall two expressions of the solution (see \cite{InMi20} for details). 

Let 
\begin{align*}
	e_{+}(t,\xi)&:=((1+t)|\xi|)^{\frac{1}{2}} J_{\nu}((1+t)|\xi|)
	\text{ and }
	e_{-}(t,\xi):=((1+t)|\xi|)^{\frac{1}{2}} Y_{\nu}((1+t)|\xi|) 
\end{align*}
where $J_{\nu}$ is the Bessel function (of the first kind)  and $Y_{\nu}$ is  the Neumann function (the Bessel function of the second type). 
Then, the solution of \eqref{eq2.1} is given by 
\begin{align}
\label{eq2.3.0}
	v(t) = \mathcal{E}_{0}(t,t_0) v_0 + \mathcal{E}_{1}(t,t_0) v_1,
\end{align}
where $\mathcal{E}_{i}(t,t_0) = \mathcal{F}^{-1} E_{i}(t,t_0,\xi) \mathcal{F}$ for $i=0,1$,
\begin{align*}
	E_{0}(t,t_0,\xi)
	&:=\frac{e_{+}(t,\xi) \dot{e_{-}}(t_0,\xi) - \dot{e_{+}}(t_0,\xi) e_{-}(t,\xi)}{e_{+}(t_0,\xi) \dot{e_{-}}(t_0,\xi) - \dot{e_{+}}(t_0,\xi) e_{-}(t_0,\xi) },
	\\
	E_{1}(t,t_0,\xi)
	&:=\frac{e_{+}(t_0,\xi) e_{-}(t,\xi) - e_{+}(t,\xi) e_{-}(t_0,\xi)}{e_{+}(t_0,\xi) \dot{e_{-}}(t_0,\xi) - \dot{e_{+}}(t_0,\xi) e_{-}(t_0,\xi) },
\end{align*}
$\mathcal{F}, \mathcal{F}^{-1}$ are the spatial Fourier transform and its inverse respectively, and $\dot{f}$ denotes the time derivative of the function $f$. 
By the vector expression, we have
\begin{align*}
	\begin{pmatrix}
	v \\ \dot{v}
	\end{pmatrix}
	= \mathcal{E}(t,t_0)
	\begin{pmatrix}
	v_0 \\ v_1
	\end{pmatrix}
\text{ where }
	\mathcal{E}(t,t_0)=
	\begin{pmatrix}
	 \mathcal{E}_{0}(t,t_0) &  \mathcal{E}_{1}(t,t_0)
	 \\
	 \dot{\mathcal{E}_{0}}(t,t_0) &  \dot{\mathcal{E}_{1}}(t,t_0)
	\end{pmatrix}.
\end{align*}

We have another expression of the solutions regarding the time dependent mass as an inhomogeneous term, namely, 
\begin{align}
\label{eq2.3}
	v(t) = \mathcal{W}_0(t-t_0) v_0 + \mathcal{W}_1(t-t_0)  v_1 + \int_{t_0}^{t} \mathcal{W}_1(t-s) \frac{-\mu}{(1+s)^2} v(s)ds,
\end{align}
where $\mathcal{W}_0:= \cos (t|\nabla|)$ and $\mathcal{W}_1:=|\nabla|^{-1}\sin (t|\nabla|)$ are the propagators of the free wave equation. 
By the vector expression, we have
\begin{align*}
	\begin{pmatrix}
	v \\ \dot{v}
	\end{pmatrix}
	= \mathcal{W}(t-t_0)
	\begin{pmatrix}
	v_0 \\ v_1
	\end{pmatrix}
	+ \int_{t_0}^{t}  \mathcal{W}(t-s) F(s,u(s)) ds
\end{align*}
where 
\begin{align*}
	\mathcal{W}(t):=
	\begin{pmatrix}
	 \mathcal{W}_{0}(t) &  \mathcal{W}_{1}(t)
	 \\
	 \dot{\mathcal{W}_{0}}(t) &  \dot{\mathcal{W}_{1}}(t)
	\end{pmatrix},
	\quad
	F(t,u) :=
	\begin{pmatrix}
	0
	\\
	-\frac{\mu}{(1+t)^2} u
	\end{pmatrix}.
\end{align*}


\subsection{Estimates of $L^2$-norm}

Recalling the estimate of the $L^2$-norm of the solution to \eqref{eq2.1} by do Nascimento, Palmieri, and Reissig \cite{NPR17}, we give a little better estimate than that in \cite{NPR17}. See also \cite{Dab15} for $L^1$-data.

\begin{lemma}[$\Sigma^{r}$-$L^2$ estimate]
\label{lem2.1}
Let $1\leq r <2$. 
If $(v_0, v_1)$ belongs to $\Sigma^{r}$, then the solution $v$ of \eqref{eq2.1} satisfies the following estimates.
\begin{align*}
	&\| v(t) \|_{L^2} \leq C_{t_0} \| (v_0,v_1) \|_{\Sigma^{r}} 
	\begin{cases}
	\{1+\log(1+t)\}^{\delta} & \text{ if } \alpha < 0, \text{ or } \alpha=0 \text{ and } r>1,
	\\
	\{1+\log(1+t)\}^{\frac{1}{2}+\delta} & \text{ if } \alpha = 0 \text{ and } r=1,
	\\
	(1+t)^{\alpha}\{1+\log(1+t)\}^{\delta} & \text{ if } \alpha > 0,
	\end{cases} 
\end{align*}
where the  constant $C_{t_0}$ is independent of $t$. 
\end{lemma}


\begin{proof}
Recall that $\widehat{v}(t,\xi)=E_0(t,t_0,\xi)\widehat{v_0}(\xi)+E_1(t,t_0,\xi)\widehat{v_1}(\xi)$. We only consider the case of $\mu\neq1/4$. 
For $t>t_0\geq 0$ and for large $N\in \mathbb{N}$, we divide $\mathbb{R}^d$ into 
\begin{align*}
	Z_1&=Z_1(t,t_0,N):=\{\xi \in \mathbb{R}^d : N \leq (1+t_0) |\xi| \},
	\\
	Z_2&=Z_2(t,t_0,N):=\{\xi \in \mathbb{R}^d :  (1+t_0) |\xi|\leq N \leq (1+t) |\xi|\},
	\\
	Z_3&=Z_3(t,t_0,N):=\{\xi \in \mathbb{R}^d :  (1+t) |\xi| \leq N \}.
\end{align*}

\noindent{\bf Case 1. Estimate in $Z_1$.} By using the $L^\infty$-estimates of $E_{0}$ and $E_{1}$ in \cite[Appendix A.3, Case 1-1]{InMi20}, we get
\begin{align*}
	&\| E_{0} (t,t_0,\xi) \widehat{v_0} \|_{L^2(Z_1)} + \| E_{1} (t,t_0,\xi) \widehat{v_1} \|_{L^2(Z_1)}
	\\
	&\cleq \| \widehat{v_0} \|_{L^2(Z_1)} + \| |\xi|^{-1}\widehat{v_1}\|_{L^2(Z_1)}
	\\
	&\cleq \| \widehat{v_0} \|_{L^2(Z_1)} +  (1+t_0)\|\widehat{v_1}\|_{L^2(Z_1)}.
\end{align*}

\noindent{\bf Case 2. Estimate in $Z_2$.} By the $L^\infty$-estimates of $E_{0}$ and $E_{1}$ in \cite[Appendix A.3, Case 1-2]{InMi20}, we have
\begin{align*}
	&\| E_{0} (t,t_0,\xi) \widehat{v_0} \|_{L^2(Z_2)} + \| E_{1} (t,t_0,\xi) \widehat{v_1} \|_{L^2(Z_2)}
	\\
	&\cleq \| |\xi|^{-\frac{1}{2}-\re\nu }\widehat{v_0} \|_{L^2(Z_2)} +(1+t_0)^{\frac{1}{2}-\re\nu} \| |\xi|^{-\frac{1}{2}-\re\nu }\widehat{v_1}\|_{L^2(Z_2)}.
\end{align*}
When $\alpha=0$ and $r>1$, by the Sobolev embedding $L^r(\mathbb{R}^d) \hookrightarrow \dot{H}^{-(1/2+\re\nu)}$, we obtain
\begin{align*}
	\| |\xi|^{-\frac{1}{2}-\re\nu }\widehat{v_0} \|_{L^2(Z_2)} +(1+t_0)^{\frac{1}{2}-\re\nu} \| |\xi|^{-\frac{1}{2}-\re\nu }\widehat{v_1}\|_{L^2(Z_2)}
	\cleq_{t_0}  \| v_0 \|_{L^{r}} + \| v_1\|_{L^{r}}
\end{align*}
In other cases, by the H\"{o}lder inequality and the Hausdorff-Young inequality, we obtain
\begin{align*}
	&\| |\xi|^{-\frac{1}{2}-\re\nu }\widehat{v_0} \|_{L^2(Z_2)} +(1+t_0)^{\frac{1}{2}-\re\nu} \| |\xi|^{-\frac{1}{2}-\re\nu }\widehat{v_1}\|_{L^2(Z_2)}
	\\
	&\cleq_{t_0} \| |\xi|^{-\frac{1}{2}-\re\nu } \|_{L^{\frac{2r}{2-r}}(Z_2)} ( \| \widehat{v_0} \|_{L^{r'}} + \|\widehat{v_1}\|_{L^{r'}})
	\\
	&\cleq_{t_0} \| |\xi|^{-\frac{1}{2}-\re\nu } \|_{L^{\frac{2r}{2-r}}(Z_2)} ( \| v_0 \|_{L^{r}} + \| v_1\|_{L^{r}}),
\end{align*}
where $r'$ denotes the H\"{o}lder conjugate exponent of $r$. 
We have
\begin{align*}
	\| |\xi|^{-\frac{1}{2}-\re\nu } \|_{L^{\frac{2r}{2-r}}(Z_2)}
	 \cleq 
	\begin{cases}
	(1+t_0)^{\alpha} & \text{ if } \alpha<0,
	\\
	\{\log(1+t)\}^{\frac{1}{2}} & \text{ if } \alpha=0 \text{ and } r=1,
	\\
	(1+t)^{\alpha} & \text{ if } \alpha>0.
	\end{cases}
\end{align*}
Therefore, we obtain
\begin{align*}
	&\| E_{0} (t,t_0,\xi) \widehat{v_0} \|_{L^2(Z_2)} + \| E_{1} (t,t_0,\xi) \widehat{v_1} \|_{L^2(Z_2)}
	\\
	&\cleq_{t_0} ( \|  v_0 \|_{L^r} + \| v_1\|_{L^r})
	\begin{cases}
	1 & \text{ if } \alpha<0, \text{ or } \alpha=0 \text{ and } r>1,
	\\
	\{\log(1+t)\}^{\frac{1}{2}} & \text{ if } \alpha=0 \text{ and } r=1,
	\\
	(1+t)^{\alpha} & \text{ if } \alpha>0.
	\end{cases}
\end{align*}

\noindent{\bf Case 3. Estimate in $Z_3$.} By the $L^\infty$-estimates of $E_{0}$ and $E_{1}$ in \cite[Appendix A.3, Case 1-3]{InMi20}, the H\"{o}lder inequality, and the Hausdorff-Young inequality, we have
\begin{align*}
	&\| E_{0} (t,t_0,\xi) \widehat{v_0} \|_{L^2(Z_3)} + \| E_{1} (t,t_0,\xi) \widehat{v_1} \|_{L^2(Z_3)}
	\\
	&\cleq (1+t)^{\frac{1}{2}+\re\nu} (\| \widehat{v_0} \|_{L^2(Z_3)} + (1+t_0)^{\frac{1}{2}-\re\nu}\| |\xi|^{-1}\widehat{v_1}\|_{L^2(Z_3)})
	\\
	&\cleq_{t_0}(1+t)^{\frac{1}{2}+\re\nu} \|1\|_{L^{\frac{2r}{2-r}}(Z_3)} (\| \widehat{v_0} \|_{L^{r'}} +\|\widehat{v_1}\|_{L^{r'}})
	\\
	&\cleq_{t_0}(1+t)^{\frac{1}{2}+\re\nu-\frac{d(2-r)}{2r}} (\|  v_0 \|_{L^r} +\| v_1\|_{L^r})
\end{align*}

Combining these estimates, we get the statement. When $\mu =1/4$, we get the logarithmic term by a modification (see \cite[Appendix A.4]{InMi20}).
\end{proof}

By the similar argument, we have the boundedness of the $\dot{H}^1 \times L^2$-norm when $(v_0,v_1) \in \Sigma^{r}$. 

\begin{lemma}
\label{lem2}
Let $1\leq r < 2$. 
If $(v_0, v_1)$ belongs to $\Sigma^{r}$, then the solution $v$ of \eqref{eq2.1} satisfies the following estimates.
\begin{align*}
	\| \overrightarrow{v}(t) \|_{\dot{H}^1 \times L^2} 
	&\leq C_{t_0} \| (v_0,v_1) \|_{\Sigma^{r}}
	\begin{cases}
	1 & \text{ if } \alpha<1,  \text{ or } \alpha=1 \text{ and } r>1,
	\\
	\{\log(1+t)\}^{\frac{1}{2}} & \text{ if } \alpha=1 \text{ and } r=1,
	\\
	(1+t)^{\alpha-1} & \text{ if } \alpha>1.
	\end{cases}
\end{align*}
where the  constant $C_{t_0}$ is independent of $t$. 
\end{lemma}

\begin{proof}
It is sufficient to consider the estimates for $|\xi|E_j(t,t_0,\xi)$ and $\dot{E}_j(t,t_0,\xi)$ for $j =0,1$. These estimates can be obtained by Lemma A.8 in \cite{InMi20} and the same method as in the proof of Lemma \ref{lem2}. We omit the detail proof.  We only note that the logarithmic term does not appear in the estimate of the $\dot{H}^1\times L^2$-norm when $\mu=1/4$. 
\end{proof}

\section{Proof of the main results}

We show the existence of the limit of $\overrightarrow{v}(t)$ in $\dot{H}^1(\mathbb{R}^d) \times L^2(\mathbb{R}^d)$.

\begin{lemma}
\label{lem3.1}
Under the assumption of Theorem \ref{thm1.0}, there exists $\overrightarrow{v_{+}} \in \dot{H}^1(\mathbb{R}^d) \times L^2(\mathbb{R}^d)$ such that
\begin{align*}
	\norm{\overrightarrow{v}(t) - \mathcal{W}(t)\overrightarrow{v_{+}}}_{\dot{H}^1 \times L^2}
	\to 0
\end{align*}
as $t \to \infty$.  
\end{lemma}

\begin{proof}
We only treat the case $\mu\neq 1/4$ and $\alpha>0$ under the assumption of Theorem \ref{thm1.0}. In other cases, the statement also holds by a small modification.
For $0\leq \tau \leq t$, by Lemma \ref{lem2.1} and $\alpha<1$, we get
\begin{align*}
	&\norm{\mathcal{W}(-t)\overrightarrow{v}(t)- \mathcal{W}(-\tau)\overrightarrow{v}(\tau)}_{\dot{H}^1\times L^2}
	\\
	&\cleq \norm{ \int_{\tau}^{t} \mathcal{W}_1(-s) \frac{\mu}{(1+s)^2} v(s)ds}_{\dot{H}^1}
	+\norm{ \int_{\tau}^{t} \mathcal{W}_0(-s) \frac{\mu}{(1+s)^2} v(s)ds}_{L^2}
	\\
	&\cleq |\mu| \int_{\tau}^{t} (1+s)^{-2} \norm{v(s)}_{L^2} ds
	\\
	&\cleq |\mu| \int_{\tau}^{t} (1+s)^{-2+\alpha}  ds  \| (v_0,v_1) \|_{\Sigma^{r}} 
	\\
	&\to 0
\end{align*}
as $\tau \to \infty$. 
By the completeness of $\dot{H}^1\times L^2$, we get the limit $\overrightarrow{v_+} \in \dot{H}^1\times L^2$.
\end{proof}

We show Theorem \ref{thm1.0}. 

\begin{proof}[Proof of Theorem \ref{thm1.0}]
It holds from Lemma \ref{lem2.1} that
\begin{align*}
	&\norm{\overrightarrow{v}(t)- \mathcal{W}(t)\overrightarrow{v_+}}_{\dot{H}^1\times L^2}
	\cleq |\mu| \int_{t}^{\infty} (1+s)^{-2} \|v(s)\|_{L^2} ds
	\\
	&\cleq |\mu| C_{t_0} \| (v_0,v_1) \|_{\Sigma^{r}}
	\\
	&\quad \times
	\begin{cases}
	(1+t)^{-1}\{1+\log(1+t)\}^{\delta} & \text{ if } \alpha < 0, \text{ or } \alpha=0 \text{ and } r>1,
	\\
	(1+t)^{-1}\{1+\log(1+t)\}^{\frac{1}{2}+\delta} & \text{ if } \alpha = 0 \text{ and } r=1,
	\\
	(1+t)^{\alpha-1}\{1+\log(1+t)\}^{\delta} & \text{ if } \alpha \in( 0,1).
	\end{cases} 
\end{align*}
This completes the proof. 
\end{proof}


Next, we give a proof of Theorem \ref{thm1.5}.

\begin{proof}[Proof of Theorem \ref{thm1.5}]
We only consider the case $\alpha(1,\mu)>1$. 
Let $(v_0,v_1) \in H^1 \times L^2$. For arbitrary $\varepsilon>0$, there exists $(v_{0,\varepsilon},v_{1,\varepsilon}) \in \Sigma^{1}$ such that $\|(v_0,v_1) - (v_{0,\varepsilon},v_{1,\varepsilon})\|_{\dot{H}^1 \times L^2} < \varepsilon$. Then, by \cite[Corollary A.9]{InMi20} and Lemma \ref{lem2} we have
\begin{align*}
	&(1+t)^{\frac{1}{2}-\re\nu } \| \overrightarrow{v}(t)\|_{\dot{H}^1 \times L^2}
	\\
	& \quad \leq (1+t)^{\frac{1}{2}-\re\nu } \| \overrightarrow{v}(t) - \overrightarrow{v_{\varepsilon}}(t) \|_{\dot{H}^1 \times L^2}
	+(1+t)^{\frac{1}{2}-\re\nu } \| \overrightarrow{v_{\varepsilon}}(t)\|_{\dot{H}^1 \times L^2} 
	\\
	&\quad \cleq  \|(v_0,v_1) - (v_{0,\varepsilon},v_{1,\varepsilon})\|_{\dot{H}^1 \times L^2} 
	+ (1+t)^{\frac{1}{2}-\re\nu } (1+t)^{\alpha-1} \| (v_{0,\varepsilon},v_{1,\varepsilon}) \|_{\Sigma^{1}} 
	\\
	&\quad \cleq \varepsilon
	+ (1+t)^{-\frac{d}{2}} \| (v_{0,\varepsilon},v_{1,\varepsilon}) \|_{\Sigma^{1}},
\end{align*}
where $\overrightarrow{v}(t)=(v(t),\partial_t v(t))$ and $\overrightarrow{v_{\varepsilon}}(t)=(v_{\varepsilon}(t),\partial_t v_{\varepsilon}(t))$ denotes the  solution to \eqref{KG} with the initial data $(v_0,v_1)$ and $(v_{0,\varepsilon},v_{1,\varepsilon})$, respectively. Thus, we obtain
\begin{align*}
	\lim_{t \to \infty} (1+t)^{\frac{1}{2}-\re\nu } \| \overrightarrow{v}(t)\|_{\dot{H}^1 \times L^2}=0. 
\end{align*}
In the other case, $\| \overrightarrow{v_{\varepsilon}}(t)\|_{\dot{H}^1 \times L^2}$ is bounded or has a logarithmic growth and thus $(1+t)^{1/2-\re\nu } \| \overrightarrow{v_{\varepsilon}}(t)\|_{\dot{H}^1 \times L^2}  \to 0$ as $t\to \infty$ since $\re \nu > 1/2$ by $\mu<0$. We get \eqref{eq01}. 
 For the $L^2$-norm of $v$, the estimate \eqref{eq02} follows from the density argument, Corollary A.9 in \cite{InMi20} (see also \cite[Theorem]{BoRe12}), and Lemma \ref{lem2.1} above. By retransforming $u=(1+t)^{\mu_1/2}v$, we obtain the estimates for \eqref{DW}. This completes the proof. 
\end{proof}

\begin{proof}[Proof of Corollary \ref{cor1.6}]
We only treat the case of $\mu \neq1/4$. 
It holds from Theorem \ref{thm1.5} that, for any $\varepsilon>0$, there exists $T=T_{\varepsilon}>0$ such that $(1+t)^{-\frac{1}{2} -\re \nu} \| v(t) \|_{L^2} <\varepsilon$ for $t>T$.
As in \cite[Proof of Theorem 1.1]{InMi20}, we have 
\begin{align*}
	\norm{\overrightarrow{v}(t)- \mathcal{W}(t)\overrightarrow{v_+}}_{\dot{H}^1\times L^2}
	\cleq \varepsilon \mu \int_{t}^{\infty} (1+s)^{-\frac{3}{2}+\re \nu}  ds
	\cleq \varepsilon \mu (1+t)^{-\frac{1}{2}+\re\nu}
\end{align*}
for $t>T$.
This gives the desired result. We can also treat the case of $\mu=1/4$ in the same way. This completes the proof. 
\end{proof}

\appendix

\section{Non-zero state}
\label{appA}

We define the energy $E$ of \eqref{KG} by 
\begin{align*}
	E(\overrightarrow{v}):= \frac{1}{2} \|\partial_t v\|_{L^2}^2 + \frac{1}{2} \|\nabla v\|_{L^2}^2 + \frac{\mu}{2(1+t)^2} \| v\|_{L^2}^2.
\end{align*}
Then we have the following proposition.

\begin{proposition}
\label{propA.1}
If $\mu<0$ and the initial data satisfies $E(\overrightarrow{v}(0))> 0$, then the solution $v$ does not decay to zero.
\end{proposition}

\begin{proof}
By using the equation \eqref{KG}, we have
\begin{align*}
	\frac{d}{dt}E(\overrightarrow{v}(t)) = -\frac{\mu}{(1+t)^3} \|v(t)\|_{L^2}^2
\end{align*}
and thus 
\begin{align*}
	E(\overrightarrow{v}(t))= \int_{0}^{t}\frac{-\mu}{(1+s)^3} \|v(s)\|_{L^2}^2 ds +E(\overrightarrow{v}(0)).
\end{align*}
It holds from $\mu<0$ that
\begin{align*}
	\frac{1}{2} \|\partial_t v (t)\|_{L^2}^2 + \frac{1}{2} \|\nabla v(t)\|_{L^2}^2 
	\geq \frac{-\mu}{2(1+t)^2} \| v(t)\|_{L^2}^2+E(\overrightarrow{v}(0)) \geq E(\overrightarrow{v}(0)).
\end{align*}
Therefore, the $\dot{H}^1 \times L^2$-norm of the solution does not decay to zero.
\end{proof}



\section{Asymptotic order for negative Sovolev spaces}
\label{appB}

We define the function space $E^{(s)}$ 
by
\begin{align*}
	E^{(s)}&:= (H^1(\mathbb{R}^d) \cap \dot{H}^{-s} (\mathbb{R}^d)) \times (L^2(\mathbb{R}^d) \cap \dot{H}^{-s}(\mathbb{R}^d))
\end{align*}
for $s > 0$. 
Moreover, we set
\begin{align*}
	\gamma&=\gamma(s,\mu):=\frac{1}{2} + \re\nu -s,
\end{align*}
where $\nu$ is defined in \eqref{eq1.1}. 

We have the following scattering and the asymptotic order for the initial data in $E^{(s)}$ with $s>\max\{-1/2+\re\nu,0\}$. 

\begin{theorem}
\label{thm1.1}
Let $d\in \mathbb{N}$, $s> \max\{-1/2+\re\nu,0\}$, and $v$ be the solution to \eqref{KG}. Assume $(v_0,v_1)\in E^{(s)}$.
Then, there exists $\overrightarrow{v_{+}} \in \dot{H}^1(\mathbb{R}^d) \times L^2(\mathbb{R}^d)$ such that
\begin{align*}
	&\norm{\overrightarrow{v}(t) - \mathcal{W}(t)\overrightarrow{v_{+}}}_{\dot{H}^1 \times L^2}
	\\
	& \quad \cleq |\mu| \norm{ (v_0,v_1) }_{E^{(s)}}
	\begin{cases}
	(1+t)^{-1}\{1+\log(1+t)\}^{\delta}& \text{ if } \gamma\leq 0,
	\\
	(1+t)^{-1+\gamma} \{1+\log(1+t)\}^{\delta} & \text{ if } \gamma>0,
	\end{cases}
\end{align*}
where the implicit constant does not depend on time and the initial data $(v_0,v_1)$.
\end{theorem}

This also gives the following result for \eqref{DW} in the same way as Corollary \ref{cor1.0}. 

\begin{corollary}
\label{cor1.1}
Let $d\in \mathbb{N}$, $s> \max\{-1/2+\re\nu,0\}$, and $u$ be the solution of \eqref{DW}. Assume $(u_0,u_1)\in E^{(s)}$. 
 Then, there exists $\overrightarrow{v_{+}} \in \dot{H}^1(\mathbb{R}^d) \times L^2(\mathbb{R}^d)$ such that
\begin{align*}
	&\norm{\overrightarrow{u}(t) -(1+t)^{-\frac{\mu_1}{2}} (\mathcal{W}(t)\overrightarrow{v_{+}})_{1} }_{\dot{H}^1}
	\\
	& \quad \cleq |\mu| \norm{ (v_0,v_1) }_{E^{(s)}}
	\begin{cases}
	(1+t)^{-1- \frac{\mu_1}{2}}\{1+\log(1+t)\}^{\delta}& \text{ if } \gamma\leq 0,
	\\
	(1+t)^{-1- \frac{\mu_1}{2}+\gamma} \{1+\log(1+t)\}^{\delta} & \text{ if } \gamma>0,
	\end{cases}
%
\end{align*}
and 
\begin{align*}
	&\norm{\overrightarrow{u}(t) -(1+t)^{-\frac{\mu_1}{2}} (\mathcal{W}(t)\overrightarrow{v_{+}})_2 }_{L^2}
	\\
	&\cleq (|\mu| +|\mu_1|)\norm{ (v_0,v_1) }_{E^{(s)}}
	\begin{cases}
	(1+t)^{-1- \frac{\mu_1}{2}}\{1+\log(1+t)\}^{\delta}& \text{ if } \gamma\leq 0,
	\\
	(1+t)^{-1- \frac{\mu_1}{2}+\gamma} \{1+\log(1+t)\}^{\delta} & \text{ if } \gamma>0.
	\end{cases}
\end{align*}
\end{corollary}

\begin{remark}
Let $(v_0,v_1)$ belong to $\Sigma^r$, $r \in (1,2)$ and $\alpha=\alpha(r,\mu)<1$. Then, taking $s=d(2-r)/(2r)$, the assumptions in Theorem \ref{thm1.1} are satisfied by the Sobolev embedding $L^r \hookrightarrow \dot{H}^{-s}$. Thus, we get Theorem \ref{thm1.0} except for the case of $r=1$ by Theorem \ref{thm1.1}. 
\end{remark}

First, we show the following lemma to obtain Theorem \ref{thm1.1}. 

\begin{lemma}[$E^{(s)}$-$L^2$ estimate]
\label{lem2.2}
Let $s\geq 0$. 
The solution $v$ of \eqref{eq2.1} satisfies the following estimates.
\begin{align*}
	&\| v(t) \|_{L^2} \leq C_{t_0} \| (v_0,v_1) \|_{E^{(s)}}
	\begin{cases}
	\{1+\log(1+t)\}^{\delta}& \text{ if } \gamma\leq 0,
	\\
	(1+t)^{\gamma} \{1+\log(1+t)\}^{\delta} & \text{ if } \gamma>0,
	\end{cases}
\end{align*}
where the constant $C_{t_0}$ is independent of $t$. 
\end{lemma}

\begin{proof}
By the same method as in the proof of Lemma \ref{lem2.1}, we obtain the result. However, we give a sketch of the proof. We only treat the case of $\mu\neq 1/4$ and $\gamma>0$. In other cases, we get the statements by small modifications.

\noindent{\bf Case 1. Estimate in $Z_1$.} This can be estimated in the similar way to Case 1 of the proof of Lemma \ref{lem2.1}.

\noindent{\bf Case 2. Estimate in $Z_2$.} By the $L^\infty$-estimates of $E_{0}$ and $E_{1}$ in \cite[Appendix A.3, Case 1-2]{InMi20} and the H\"{o}lder inequality, we get
\begin{align*}
	&\| E_{0} (t,t_0,\xi) \widehat{v_0} \|_{L^2(Z_2)} + \| E_{1} (t,t_0,\xi) \widehat{v_1} \|_{L^2(Z_2)}
	\\
	&\cleq \| |\xi|^{-\frac{1}{2}-\re\nu }\widehat{v_0} \|_{L^2(Z_2)} +(1+t_0)^{\frac{1}{2}-\re\nu} \| |\xi|^{-\frac{1}{2}-\re\nu }\widehat{v_1}\|_{L^2(Z_2)}
	\\
	&\cleq_{t_0} (1+t)^{\gamma} (\| v_0 \|_{\dot{H}^{-s}} +\|  v_1 \|_{\dot{H}^{-s}}).
\end{align*}

\noindent{\bf Case 3. Estimate in $Z_3$.} By the $L^\infty$-estimates of $E_{0}$ and $E_{1}$ in \cite[Appendix A.3, Case 1-3]{InMi20} and the H\"{o}lder inequality, we have
\begin{align*}
	&\| E_{0} (t,t_0,\xi) \widehat{v_0} \|_{L^2(Z_3)} + \| E_{1} (t,t_0,\xi) \widehat{v_1} \|_{L^2(Z_3)}
	\\
	&\cleq (1+t)^{\frac{1}{2}+\re\nu} (\| \widehat{v_0} \|_{L^2(Z_3)} + (1+t_0)^{\frac{1}{2}-\re\nu}\|\widehat{v_1}\|_{L^2(Z_3)})
	\\
	&\cleq_{t_0}(1+t)^{\frac{1}{2}+\re\nu}  \||\xi|^{s}\|_{L^{\infty}(Z_3)} (\|  v_0 \|_{\dot{H}^{-s}} +\|v_1\|_{\dot{H}^{-s}})
	\\
	&\cleq_{t_0}(1+t)^{\gamma} (\|  v_0 \|_{\dot{H}^{-s}} +\| v_1 \|_{\dot{H}^{-s}})
\end{align*}

Combining these estimates, we get the statement. When $\mu =1/4$, we get the logarithmic term by a modification (see \cite[Appendix A.4]{InMi20}).
\end{proof}

\begin{proof}[Proof of Theorem \ref{thm1.1}]
By the assumption $s> \max\{-1/2+\re\nu,0\}$, we have $\gamma<1$. Therefore, in the same way as the proofs of Lemma \ref{lem3.1} and Theorem \ref{thm1.0}, we get the desired statement. 
\end{proof}


\begin{acknowledgement}
The first author is supported by by JSPS KAKENHI Grant-in-Aid for Early-Career Scientists JP18K13444 and the second author is supported by JSPS KAKENHI  Grant-in-Aid for Young Scientists (B) JP17K14218 and, partially, for Scientific Research (B) JP17H02854. 
\end{acknowledgement}


\end{document}